\documentclass[12pt]{amsart}
\usepackage{latexsym, amssymb, amsmath}

\newtheorem{theorem}{Theorem}[section]
\newtheorem{lemma}[theorem]{Lemma}

\newtheorem{remark}[theorem]{Remark}
\theoremstyle{definition}

\newtheorem*{example}{Example}

\makeatletter
    
    \@addtoreset{equation}{section}
  \makeatother

\newcommand{\n}{\nabla}

\begin{document}

\title[Low-dimensional complete gradient Yamabe solitons]
{Classification of low-dimensional complete gradient Yamabe solitons}

\author{Shun Maeta}
\address{Department of Mathematics, Chiba University, 1-33, Yayoicho, Inage, Chiba, 263-8522, Japan.}
\curraddr{}
\email{shun.maeta@faculty.gs.chiba-u.jp~{\em or}~shun.maeta@gmail.com}
\thanks{The author is partially supported by the Grant-in-Aid for Scientific Research (C), No.23K03107, Japan Society for the Promotion of Science.}


\dedicatory{}

\keywords{gradient Yamabe solitons; the Yamabe soliton version of the Perelman conjecture; Ricci solitons}

\commby{}

\begin{abstract}
In this paper, we completely classify nontrivial non-flat two- and three-dimensional complete gradient Yamabe solitons. 
\end{abstract}

\maketitle

\bibliographystyle{amsplain}


\section{Introduction}\label{intro}

The geometric flow is one of the most powerful tools for understanding the structure of Riemannian manifolds. 
In particular, the Yamabe flow is one of the central fields of the theory and has developed rapidly (cf. \cite{Brendle05}, \cite{Brendle07}). Gradient Yamabe solitons are self-similar solutions of the Yamabe flow and expected to be a singularity model. 
Therefore, the classification problem is one of the most important ones. To classify gradient Yamabe solitons, there are many studies with curvature assumptions and locally conformally flat conditions (cf. \cite{DS13}, \cite{CMM12}, \cite{CSZ12} and \cite{Maeta21}). These studies give affirmative partial answers to the Yamabe soliton version of the Perelman conjecture: {\it Under what conditions do nontrivial complete gradient shrinking, steady, and expanding Yamabe solitons with positive (scalar) curvature or with bounded (scalar) curvature become rotationally symmetric? In particular, what natural geometric assumptions, beyond positive or bounded (scalar) curvature, are required?} C. He resolved for steady case (cf. \cite{He11}).
For shrinking case, the author resolved the conjecture (cf. \cite{Maeta28}).

In order to gain a deeper understanding of gradient Yamabe solitons, we consider complete gradient Yamabe solitons without assuming any curvature assumptions.
In this paper, we completely classify nontrivial non-flat three-dimensional complete gradient Yamabe solitons. 
We also provide examples of complete expanding gradient Yamabe solitons, which do not appear in the steady and shrinking cases.

The classification of two-dimensional complete gradient Ricci solitons was previously studied by Bernstein and Mettler \cite{BM15}, whose approach relies on specific two-dimensional Killing vector fields to reduce the problem to an ordinary differential equation. 

In this paper, we provide an essentially different proof from the broader perspective of Yamabe solitons. Instead of using Killing vector fields, we determine the warped product structure from a more general framework. This approach allows us to describe the geometric properties more explicitly. 
While expanding solitons in \cite{BM15} are classified mainly by their asymptotic behaviors, our theorem provides explicit bounds for the Gaussian curvature $K$ using the soliton constant $\lambda$, proving its strict monotonicity and specific ranges such as $\frac{\lambda}{2} < K < 0$, or $0 < K \le \frac{1+\lambda}{2}$, or $\frac{1+\lambda}{2}\leq K<0$.

\begin{remark}
The original Perelman conjecture \cite{Perelman1} is that any three-dimensional complete noncompact $\kappa$-noncollapsed steady gradient Ricci soliton with positive curvature is rotationally symmetric, which was proven by S. Brendle \cite{Brendle13} (see also \cite{Brendle14}$)$.
\end{remark}

\section{Preliminary}\label{pre}

An $n$-dimensional Riemannian manifold $(M^n,g)$ is called a gradient Yamabe soliton if there exists a smooth function $F$ on $M$ and a constant $\lambda\in \mathbb{R}$, such that 
$\nabla \nabla F=(R-\lambda)g,$
where $\nabla\nabla F$ is the Hessian of $F$, and $R$ is the scalar curvature of $M$. 
If $F$ is constant, then $M$ is called trivial.
If $\lambda>0$, $\lambda=0$, or $\lambda<0$, then the Yamabe soliton is called shrinking, steady, or expanding, respectively.

Cao, Sun, and Zhang showed the structure theorem for Yamabe solitons (see also \cite{Tashiro65} and \cite{CMM12}).

\begin{theorem}[\cite{Tashiro65}, \cite{CSZ12}, \cite{CMM12}]\label{CSZ}
Let $(M,g,F)$ be a nontrivial, complete, gradient Yamabe soliton. 
Then $(M,g)$ is a complete warped product manifold and must take one of the two forms:

\begin{enumerate}
\item[$(1)$]
a warped product manifold
$([0,+\infty),dr^2)\times_{{F'(r)}^2}(\mathbb{S}^{n-1},{\bar g}_{S})$,
where $\bar g_{S}$ is the round metric on $\mathbb{S}^{n-1},$ or
\item[$(2)$]
a warped product manifold
$(\mathbb{R},dr^2)\times_{{F'(r)}^2} \left(N^{n-1},\bar g\right),$
where the scalar curvature $\bar R$ of $N$ is constant. If the scalar curvature of $M$ is nonnegative, then $
\bar R>0$ or $R=\bar R=0.$ 
\end{enumerate}
\end{theorem}

\begin{remark}\label{depends r}
~
\begin{enumerate}

\item[$(1)$]
For a conformal soliton, that is, a Riemannian manifold $(M,g,F,\varphi)$ that satisfies the condition $\nabla\nabla F=\varphi g$, where $\varphi\in C^\infty(M)$, Tashiro \cite{Tashiro65} provided the structure theorem, and Catino, Mantegazza, and Mazzieri also provided the structure theorem in \cite{CMM12} $($see also \cite{Maeta21}$).$ 
Such manifolds were also studied by Cheeger and Colding \cite{CC96}.

\item[$(2)$] It is shown that $F$ depends only on $r$, and in Case (2) of Theorem \ref{CSZ}, without loss of generality, we can assume that $\rho(r) = F'(r) > 0$ on $\mathbb{R}$ (cf. \cite{CSZ12} see also \cite{Maeta21}).
\end{enumerate}
\end{remark}

The Riemannian curvature tensor is defined by 
\[
R(X,Y)Z=-\n_X\n_YZ+\n_Y\n_XZ+\n_{[X,Y]}Z,
\]
for $X,Y,Z\in \mathfrak{X}(M)$, where $\nabla$ is the Levi-Civita connection of $M$. The Ricci tensor $R_{ij}$ is defined by 
$R_{ij}=R_{ipjp},$ where $R_{ijk\ell}=g(R({\partial_i,\partial_j})\partial_k,\partial_\ell).$

The following lemma (cf. Lemma 2.4 \cite{Maeta28}) will play a fundamental role in the proof of our results later.

\begin{lemma}\label{equivlambda}
Let $(M,g,F,\lambda)$ be a nontrivial, complete, gradient Yamabe soliton such that 
$M=\mathbb{R}\times N^{n-1}$ and $g=dr^2+{{F'(r)}^2} \bar g$.
If $R\geq\lambda$ (resp. $R\leq\lambda$) on $M$, then $R>\lambda$ or $R\equiv\lambda$ (resp. $R<\lambda$ or $R\equiv\lambda$) on $M$.
\end{lemma}


\section{Classification of three-dimensional complete shrinking gradient Yamabe solitons}\label{3sh}

In this section, we classify nontrivial, nonflat, three-dimensional complete shrinking gradient Yamabe solitons. In particular, we show that any nontrivial, nonflat, three-dimensional complete shrinking gradient Yamabe soliton is rotationally symmetric. 

\begin{theorem}\label{mainshr}
Let $(M,g,F,\lambda)$ be a nontrivial, nonflat, three-dimensional complete shrinking gradient Yamabe soliton.
Then $(M,g,F,\lambda)$ is either
\begin{enumerate}
\item[$(1)$]
$([0,+\infty),dr^2)\times_{|\nabla F|}(\mathbb{S}^{2},{\bar g}_{S})$, or
\item[$(2)$]
$(\mathbb{R},dr^2)\times_{|\nabla F|}(\mathbb{S}^{2},{\bar g}_{S}).$
\end{enumerate}

\end{theorem}

\begin{proof}

\noindent
Case (2) of Theorem \ref{CSZ}.

In this case, for $a,b,c,d=2,3,\cdots,n,$ the curvature tensors are obtained as follows:
\begin{align}\label{RT1}
R_{1a1b}&=-F'F'''{\bar g}_{ab},\quad R_{1abc}=0,\\
R_{abcd}&=(F')^2{\bar R}_{abcd}+(F'F'')^2(\bar g_{ad}\bar g_{bc}-\bar g_{ac}\bar g_{bd}),\notag
\end{align}
\begin{align*}
R_{11}=&-(n-1)\frac{F'''}{F'},\quad 
R_{1a}=0,\\
R_{ab}=&\bar R_{ab}-((n-2)(F'')^2+F'F''')\bar g_{ab},\notag
\end{align*}
\begin{align}\label{RT3}
R=(F')^{-2}\bar R-(n-1)(n-2)\Big(\frac{F''}{F'}\Big)^2-2(n-1)\frac{F'''}{F'},
\end{align}
where the curvature tensors with bar are the curvature tensors of $(N,\bar g)$.

By \eqref{RT3} and the soliton equation $\rho'=R-\lambda$, one has 
\begin{equation}\label{shr1}
\rho^2\rho'+\lambda\rho^2+2\rho'^2+4\rho\rho''-\bar R=0,
\end{equation}
where $\rho(r)=F'(r)$.
Assume that $R\geq 0$ on $\mathbb{R}$, that is, $\rho'\geq -\lambda$. By \eqref{shr1}, we have
\[
\rho''\leq\frac{\bar R-2\rho'^2}{4\rho}.
\]
If $\bar R\leq 0$, then the positive function $\rho$ is concave. Hence, $\rho$ is constant. By \eqref{shr1}, we have $\bar R=\lambda \rho^2>0,$ which is a contradiction. Therefore, $\bar R$ is positive.

Therefore, we only need to consider the other case, that is, $\rho'<-\lambda$ on some interval $(r_1,r_2)$.
If $r_2=+\infty$, then $\rho=0$ at some point, which cannot occur because $\rho>0$. Hence, $r_2<+\infty$, that is, $\rho'=-\lambda$ at $r_2.$ By \eqref{shr1}, we have $\rho''=\frac{\bar R -2\lambda^2}{4\rho}$ at $r_2$. 
If $\bar R\leq 0$, then $\rho''=\frac{\bar R -2\lambda^2}{4\rho}<0$ at $r_2$, which cannot occur because $\rho'<-\lambda$ on $(r_1,r_2)$.
Therefore, $\bar R$ is positive. 
\end{proof}


\section{Classification of three-dimensional complete steady gradient Yamabe solitons}\label{3st}

In this section, we classify three-dimensional complete steady gradient Yamabe solitons. In particular, we show that any nontrivial, nonflat, three-dimensional complete steady gradient Yamabe soliton is rotationally symmetric.

\begin{theorem}\label{mainste}
Let $(M,g,F)$ be a nontrivial, nonflat, three-dimensional complete steady gradient Yamabe soliton.
Then $(M,g,F)$ is one of the following:
\begin{enumerate}
\item[$(1)$]
$([0,+\infty),dr^2)\times_{|\nabla F|}(\mathbb{S}^{2},{\bar g}_{S})$, or
\item[$(2)$]
$(\mathbb{R},dr^2)\times_{|\nabla F|}(\mathbb{S}^{2},{\bar g}_{S})$. In this case, there exists at least one point such that $R=0$.
\end{enumerate}
\end{theorem}

\begin{proof}
\noindent
Case (2) of Theorem \ref{CSZ}.

Since $(N^2,{\overline g})$ is a two-dimensional manifold,
\begin{align}\label{RT4}
{\overline R}_{abcd}&=-\frac{\bar R}{2}(\bar g_{ad} \bar g_{bc}-\bar g_{ac}\bar g_{bd}).
\end{align}
By \eqref{RT3} and the soliton equation $R=\rho'$, one has
\begin{equation}\label{ste1}
\rho^2\rho'+2\rho'^2+4\rho\rho''=\bar R,
\end{equation}
where $\rho(r)=F'(r).$
By differentiating both sides of \eqref{ste1}, we have
\begin{equation}\label{ste2}
2\rho\rho'^2+\rho^2\rho''+8\rho'\rho''+4\rho\rho'''=0.
\end{equation}

Assume that $\rho'\geq0$ on $\mathbb{R}$. If $\rho''\leq 0$ on $\mathbb{R}$, then the positive function $\rho$ is concave. Hence, $\rho$ is constant. By \eqref{ste1}, $\bar R=0$. By \eqref{RT4} and \eqref{RT1}, $M$ is flat.
Therefore, $\rho''>0$ on some interval $(r_1,r_2)$. 

If $r_2=+\infty$, since $\rho'\geq0$, $\rho$ tends to infinity as $r\nearrow+\infty$. 
Hence, the left-hand side of \eqref{ste1} tends to infinity as $r\nearrow+\infty$, which cannot occur.

Assume that $r_2<+\infty$, that is, $\rho''>0$ on $(r_1,r_2)$ and $\rho''=0$ at $r_2.$
By \eqref{ste2}, we have $\rho'''<0$ at $r_2$. 
Hence, one has $\rho''< 0$ on $(r_2,r_3)$ for some $r_3$. By the same argument, $\rho'$ is monotonically decreasing on $(r_2,+\infty)$. Without loss of generality, we can assume that the point $r_2$ is the first critical point of $\rho'$.
Since $\rho'\geq0$ and $\rho''>0$ on $(-\infty, r_2)$, there exists $r_0\in(-\infty,r_2)$ such that $\rho'''=0$. However, this contradicts \eqref{ste2}.
 
Therefore, we have $\rho'<0$ on some open interval $(r_1,r_2)$.

Assume that $r_2= +\infty$, that is, $\rho'<0$ on $(r_1,+\infty)$.
If $\rho''<0$ on $(r_1,+\infty)$, we have a contradiction, because $\rho>0$.

Assume that there exists $\tilde r\in(r_1,+\infty)$ such that $\rho''(\tilde r)=0$. By \eqref{ste2}, we have $\rho'''(\tilde r)<0$. Iterating this argument, we see that $\rho'$ is weakly decreasing on $(\tilde r,+\infty)$. Hence, one has $\rho'<-\beta^2$ for some nonzero constant $\beta$. 
This is a contradiction, because $\rho>0$.

Hence, $\rho''>0$ on $(r_1,+\infty)$. 
Thus, one has $\rho'(r)\nearrow \alpha(\leq0)$ for some constant $\alpha$ as $r\nearrow+\infty.$ If $\alpha<0$, then we have a contradiction, because $\rho>0$. Hence, one has $\rho'\nearrow0$ as $r\nearrow+\infty.$
 By \eqref{ste1}, $\rho''$ converges to some constant $c$ or tends to $+\infty$ as $r\nearrow+\infty$. Since $\rho>0$ and $\rho'<0$ on $(r_1,+\infty)$, the second case cannot occur, and we have the first case with $c=0$, that is, $\rho''\searrow0$ as $r\nearrow +\infty.$ 
 By \eqref{ste1}, $\bar R$ must be $0$. 
 Assume that $(r_{-2},r_{-1})$ is some maximal interval such that $\rho'>0$.
 By \eqref{ste1}, we have $\rho''<0$ on $(r_{-2},r_{-1})$. If $\rho''=0$ at some point, then by \eqref{ste1}, we have $\rho'\leq0$ at this point, which means that $r_{-2}$ must be $-\infty$. However, since $\rho>0$, this cannot occur.
 Hence, we have $\rho'\leq0$ on $\mathbb{R}$. Therefore, $R\leq0$ on $M.$

Assume that $r_2<+\infty$, that is, $\rho'=0$ at $r_2$. 

\noindent
Case 1: $\bar R>0$. In this case, 
\[
\rho''=\frac{\bar R}{4\rho}>0,
\]
at $r_2$, and hence
$\rho'>0~~\text{on}~~(r_2,r_3),~~~\text{for some}~~r_3.$ 
If $\rho'=0$ at $r_3$, then $\rho''>0$ at $r_3$. Hence, $\rho$ is monotonically increasing on $(r_2,+\infty)$.
If $\rho'=0$ at $r_1$, then $\rho''>0$ at $r_1$, which cannot occur.
Therefore, $\rho'<0$ on $(-\infty, r_2)$. 

\noindent
Case 2: $\bar R=0$. By the same argument as above, we have $\rho'\leq 0$ on $\mathbb{R}$.
The equation \eqref{ste1} becomes
\[
4\rho\rho'' + 2(\rho')^2 + \rho^2\rho' = 0.
\]
This is equivalent to
\[
\frac{d}{dr}\left( 2\rho^{1/2}\rho' + \frac{1}{5}\rho^{5/2} \right) = 0.
\]
Integrating this with respect to $r$ yields
\[
2\rho^{1/2}\rho' + \frac{1}{5}\rho^{5/2} = C,
\]
where $C$ is a constant of integration. Hence, we obtain
\[
\rho' = -\frac{1}{10}\rho^2 + \frac{C}{2}\rho^{-1/2}.
\]
We now consider the sign of the constant $C$.

\noindent
{Case 2-(a): $C \le 0$.}

Since $\rho > 0$, we have $C\rho^{-1/2} \le 0$. Thus, the equation implies
\[
\rho' \le -\frac{1}{10}\rho^2.
\]
Dividing by $\rho^2$, we obtain $-\frac{\rho'}{\rho^2} \ge \frac{1}{10}$.
Integrating this inequality from $r$ to $r_0$ (where $r < r_0$), we find
\[
\frac{1}{\rho(r_0)} - \frac{1}{\rho(r)} \ge \frac{1}{10}(r_0 - r).
\]
Taking the limit as $r \searrow -\infty$, the right-hand side diverges to $+\infty$. However, since $\rho(r) > 0$, the left-hand side is strictly bounded from above by $\frac{1}{\rho(r_0)}$. This is a contradiction.

\noindent
Case 2-(b): $C > 0$.

By the condition $\rho' \le 0$ on $\mathbb{R}$, we have
\[
-\frac{1}{10}\rho^2 + \frac{C}{2}\rho^{-1/2} \le 0,
\]
which implies 
\[
\rho(r) \ge (5C)^{2/5} > 0.
\]
This shows that $\rho(r)$ is bounded from below by a positive constant.
Since $\rho(r)$ is a monotonically decreasing function, $\rho'\nearrow0$ and $\rho\searrow\alpha \ge (5C)^{2/5}$ as $r \nearrow +\infty$.
Substituting this into the equation $\rho' = -\frac{1}{10}\rho^2 + \frac{C}{2}\rho^{-1/2}$, we find that the limit must be $\alpha = (5C)^{2/5}$.

Next, we consider the limit as $r \searrow -\infty$.
Because $\rho(r)$ is monotonically decreasing, as $r \searrow -\infty$, $\rho(r)$ either converges to a finite limit or tends to $+\infty$.
If $\rho(r)$ converges to a finite limit, by the same argument, $\rho'(r) \nearrow 0$ as $r\searrow-\infty$ and the limit must also be $(5C)^{2/5}$.
Since $\rho$ is a monotonic function, $\rho$ is constant and $M$ is flat, which contradicts the assumption that the soliton is nonflat.

Thus, $\rho(r)$ must tend to $+\infty$ as $r \searrow -\infty$.
As $\rho \nearrow +\infty$, the term $C\rho^{-1/2}$ tends to $0$.
Therefore, there exists a point $r_1 \in \mathbb{R}$ such that for all $r < r_1$, we have $\frac{C}{2}\rho^{-1/2} < \frac{1}{20}\rho^2$.
For $r < r_1$, the differential equation yields
\[
\rho' = -\frac{1}{10}\rho^2 + \frac{C}{2}\rho^{-1/2} < -\frac{1}{10}\rho^2 + \frac{1}{20}\rho^2 = -\frac{1}{20}\rho^2,
\]
which yields
\[
-\frac{\rho'}{\rho^2} > \frac{1}{20}.
\]
Integrating this inequality from $r$ to $r_1$ (where $r < r_1$), we have
\[
\int_{r}^{r_1} \left( -\frac{\rho'(s)}{\rho(s)^2} \right) ds > \int_{r}^{r_1} \frac{1}{20} ds,
\]
which gives
\[
\frac{1}{\rho(r_1)} - \frac{1}{\rho(r)} > \frac{1}{20}(r_1 - r).
\]
Taking the limit as $r \searrow -\infty$, the right-hand side tends to $+\infty$.
However, the left-hand side is strictly bounded by $\frac{1}{\rho(r_1)}$. This is a contradiction.

\noindent
Case 3: $\bar R<0$. In this case, $\rho''<0$ at $r_2$. Hence, $\rho'<0$ on $(r_2,r_3)$ for some $r_3$. 
If $r_3=+\infty,$ by the same argument as above, one has $\bar R=0$, which is a contradiction. Hence, $r_3<+\infty$, that is, $\rho'=0$ at $r_3$, which means that there exists $\tilde r_2\in(r_2,r_3)$ such that $\rho''=0$. By \eqref{ste2}, $\rho'''<0$ at $\tilde r_2$, which cannot occur.
\end{proof}


\section{Classification of three-dimensional complete expanding gradient Yamabe solitons}\label{3ex}

In the study of Ricci solitons, much less is understood about the expanding case compared to the shrinking or steady cases. 
A similar phenomenon appears to occur for Yamabe solitons as well.
However, as will become evident in the two-dimensional case discussed later, there are numerous expanding Yamabe solitons, making them an interesting subject from a geometric standpoint.
In this section, we classify nontrivial, nonflat, three-dimensional complete expanding gradient Yamabe solitons. 

\begin{theorem}\label{mainexp}
Let $(M,g,F,\lambda)$ be a nontrivial, nonflat, three-dimensional complete expanding gradient Yamabe soliton.
Then $(M,g,F,\lambda)$ is one of the following:
\begin{enumerate}
\item[$(1)$]
$([0,+\infty),dr^2)\times_{|\nabla F|}(\mathbb{S}^{2},{\bar g}_{S})$, or
\item[$(2)$]
$(\mathbb{R},dr^2)\times_{|\nabla F|}(\mathbb{S}^{2},{\bar g}_{S}),$ or
\item[$(3)$]
$(\mathbb{R},dr^2)\times_{|\nabla F|}\mathbb{E}^2$ with negative scalar curvature, or
\item[$(4)$]
$(\mathbb{R},dr^2)\times_{|\nabla F|}(\mathbb{H}^{2},{\bar g}_{H})$ with negative scalar curvature, 
\end{enumerate}
where $\mathbb{E}^2$ is the $2$-dimensional Euclidean space, and $(\mathbb{H}^{2},{\bar g}_{H})$ is a $2$-dimensional hyperbolic space.

\end{theorem}

\begin{proof}

\noindent
Case (2) of Theorem \ref{CSZ}.

By \eqref{RT3} and the soliton equation $\rho'=R-\lambda$, one has 
\begin{equation}\label{exp1}
\rho^2R+2(R-\lambda)^2+4\rho R'-\bar R=0.
\end{equation}
By differentiating both sides of \eqref{exp1}, we have
\begin{equation}\label{exp2}
2\rho(R-\lambda)R+\rho^2R'+8(R-\lambda)R'+4\rho R''=0.
\end{equation}
We consider the case $\bar R\leq 0$.
Assume that $R>0$ on some open interval $(r_1,r_2)$. In this case, by \eqref{exp1}, one has $R'<0$ on $(r_1,r_2)$. 
If $R=0$ at $r_1$, then by \eqref{exp1}, one has $R'<0$ at $r_1$, which is a contradiction. Hence, $R>0$ and $R'<0$ on $(-\infty,r_2)$. Therefore, $\rho$ is a strictly monotonically increasing concave function on $(-\infty,r_2)$. However, since $\rho$ is positive, we have a contradiction. Therefore, we have $R(r)\leq0$ on $\mathbb{R}.$ 
Assume that $R=0$ at some point $r_0$. By \eqref{exp1}, 
one has $R'<0$ at $r_0$. By \eqref{exp2}, one has $R''>0$ at $r_0$, which cannot occur. Therefore, we have $R<0$ on $\mathbb{R}.$
\end{proof}

We provide examples of (4) in Theorem \ref{mainexp}. For $n$-dimensional complete expanding gradient Yamabe solitons, an elementary argument yields interesting examples that do not appear in the steady and shrinking cases.

\begin{example}
Let $(N^{n-1},~\bar g)$ be an $(n-1)$-dimensional complete Riemannian manifold with constant negative scalar curvature $\bar R$. Then, for any $\alpha \in \mathbb{R}$, $(M,g,F,\lambda)=(\mathbb{R}\times N^{n-1}, dr^2+\frac{\bar R}{\lambda}\bar g, \sqrt{\frac{\bar R}{\lambda}} r+\alpha,\lambda)$ is an $n$-dimensional complete expanding gradient Yamabe soliton with $R=\lambda.$

In particular, $(M^3,g,F,\lambda)=(\mathbb{R}\times \mathbb{H}^2, dr^2+\frac{\bar R}{\lambda} {\bar g}_{H}, \sqrt{\frac{\bar R}{\lambda}} r+\alpha,\lambda)$ is a 3-dimensional complete expanding gradient Yamabe soliton with $R=\lambda$, where $(\mathbb{H}^2,{\bar g}_H)$ is a hyperbolic surface.

\end{example}



\section{Classification of two-dimensional complete gradient Yamabe solitons}\label{2D}

In this section, we provide a proof for the classification of nontrivial, nonflat, two-dimensional complete gradient Yamabe solitons.

\begin{theorem}\label{main2D}
Let $(M^2,g,F,\lambda)$ be a nontrivial, nonflat, two-dimensional complete gradient Yamabe soliton. Then exactly one of the following statements holds:

\noindent
{\bf Steady case:}

The manifold $(M^2,g,F,\lambda)$ is Hamilton's cigar soliton.

\noindent
{\bf Shrinking case:}

There are no nontrivial, nonflat, two-dimensional complete gradient Yamabe solitons.

\noindent
{\bf Expanding case:}

The manifold $(M^2,g,F,\lambda)$ is one of the following.
\begin{enumerate}
\item[$(1)$]
A warped product manifold $(\mathbb{R}\times N^1, g=dr^2+F'(r)^2d\theta^2,F,\lambda)$. The Gaussian curvature $K$ of $\mathbb{R}\times N^1$ depends only on $r$, is strictly increasing, and satisfies $\frac{\lambda}{2}<K<0$. Furthermore, $M$ is asymptotically flat and $K\searrow\frac{\lambda}{2}$ as $r\searrow-\infty.$

\item[$(2)$]
A warped product manifold $([0,+\infty)\times \mathbb{S}^1, g=dr^2+F'(r)^2d\theta^2,F,\lambda)$ with $\lambda>-1$, $F'(0)=0$, and $F''(0)=1$. The Gaussian curvature $K$ of $[0,+\infty)\times \mathbb{S}^1$ depends only on $r$, is strictly decreasing, and satisfies $0<K\leq\frac{1+\lambda}{2}$. Furthermore, $M$ is asymptotically flat.

\item[$(3)$]
A warped product manifold $([0,+\infty)\times \mathbb{S}^1, g=dr^2+F'(r)^2d\theta^2,F,\lambda)$ with $\lambda<-1$, $F'(0)=0$, and $F''(0)=1$. The Gaussian curvature $K$ of $[0,+\infty)\times \mathbb{S}^1$ depends only on $r$, is strictly increasing, and satisfies $\frac{1+\lambda}{2}\leq K<0$. Furthermore, $M$ is asymptotically flat.
\end{enumerate}

\end{theorem}

\begin{proof}
By the soliton equation, one has
\[
F''(r)=R-\lambda
.\]
By an elementary fact of the curvature of a warped product, one has
\[
R=-2\frac{F'''}{F'}.
\]
Combining these equations, we have
\begin{equation}\label{1}
2\rho''+\rho\rho'+\lambda\rho=0,
\end{equation}
where we denote $\rho(r):=F'(r).$ 
By differentiating both sides of \eqref{1}, we have
\begin{equation}\label{3}
2\rho'''+\rho'^2+\rho\rho''+\lambda\rho'=0.
\end{equation}
The equation \eqref{1} is equivalent to
\begin{equation*}
2R'+\rho R=0.
\end{equation*}

We show that $R>0$ or $R<0$ on $\mathbb{R}$ (resp. $[0,+\infty)$).
Assume that there exists $\tilde r \in \mathbb{R}$ (resp. $\tilde r\in[0,+\infty)$) such that $R(\tilde r)=0$.
Set $A(r)=\int_{\tilde r}^r \frac{\rho(s)}{2}ds.$
Since 
\[
\frac{d}{dr}(e^{A(r)}R(r))
=e^{A(r)}
\left(
\frac{1}{2}\rho(r) R(r)+R'(r)
\right)
=0,
\]
we see that $e^{A(r)}R(r)$ is constant, say $c$. Thus, $R(r)=ce^{-A(r)}.$ Since $A(\tilde r)=0$, we have $c=R(\tilde r)=0$. Therefore, we have $R\equiv0$ on $\mathbb{R}$ (resp. $[0,+\infty)$). Hence, $M$ is flat.
 \\
 
\noindent
{\bf Case (2) of Theorem \ref{CSZ}.}
\\

\noindent
{\bf Case 1 $\lambda=0$.}

In this case, it is easy to see that the solutions to \eqref{1} are
\[
\rho(r)=2c_1^2\tanh\left\{\frac{1}{2}(c_1^2 r+c_1^2c_2 )\right\},
\]
where $c_1$ and $c_2$ are constants. However, this contradicts the assumption that $\rho(r)>0$ on $\mathbb{R}$.
\\

\noindent
{\bf Case 2 $\lambda>0$.}
\\

\noindent
(I) When $R>0$. By \eqref{1}, one has $\rho''<0$. Therefore, $\rho$ is a positive concave function, which cannot occur.
\\

\noindent
(II) When $R<0$. We have $\rho'<-\lambda$. Since $\rho>0$, we have a contradiction. 
\\

\noindent
{\bf Case 3 $\lambda<0$}.
\\

\noindent
(I) When $R>0$. By \eqref{1}, one has $\rho''<0$. Therefore, $\rho$ is a positive concave function, which cannot occur.
\\

\noindent
(II) When $R<0$. By \eqref{1}, $\rho''>0$. Hence, there exists at most one $r_0\in\mathbb{R}$ such that $\rho'(r_0)=0$.
\\

\noindent
\noindent
(II)-(a) Assume that there exists $r_0 \in \mathbb{R}\cup\{+\infty\}$ such that $\rho' < 0$ on $(-\infty, r_0)$. 
(This covers both the case where $\rho'(r_0) = 0$ for some finite $r_0$, and the case where $\rho' < 0$ on all of $\mathbb{R}$.)
Because $\rho$ is positive, strictly decreasing, and convex on $(-\infty, r_0)$, we have $\rho(r)\nearrow+\infty$ as $r\searrow-\infty$.
By \eqref{1}, we have $2\rho'' = -\rho\rho' - \lambda\rho$. Since $\lambda < 0$ and $\rho > 0$, we obtain the strict inequality
\begin{equation*}
2\rho'' > -\rho\rho'.
\end{equation*}
Integrating this inequality from $r$ to $r_1$ (where $r < r_1 < r_0$), we obtain
\begin{equation*}
2\rho'(r_1) - 2\rho'(r) > -\frac{1}{2}\rho(r_1)^2 + \frac{1}{2}\rho(r)^2.
\end{equation*}
Thus, we have
\begin{equation*}
-2\rho'(r) > \frac{1}{2}\rho(r)^2 - C,
\end{equation*}
where $C = \frac{1}{2}\rho(r_1)^2 + 2\rho'(r_1)$ is a constant.
Since $\rho(r) \nearrow +\infty$ as $r \searrow -\infty$, for $r$ sufficiently negative (say, $r < r_2 < r_1$), we have $\frac{1}{2}\rho(r)^2 - C > \frac{1}{4}\rho(r)^2 > 0$.
Thus, we get $-2\rho'(r) > \frac{1}{4}\rho(r)^2$, which can be rewritten as
\begin{equation*}
-\frac{\rho'(r)}{\rho(r)^2} > \frac{1}{8}.
\end{equation*}
Integrating this inequality from $r$ to $r_2$ yields
\begin{equation*}
\frac{1}{\rho(r_2)} - \frac{1}{\rho(r)} > \frac{1}{8}(r_2 - r).
\end{equation*}
Letting $r\searrow-\infty$, we have a contradiction.
\\

\noindent 
(II)-(b) Assume that $\rho'>0$ on $\mathbb{R}$.
Since $\rho'>0$ and $\rho''>0$, we have $\rho\nearrow+\infty$ as $r\nearrow+\infty.$ We show $\rho'\nearrow-\lambda$ as $r\nearrow+\infty.$ If $\rho'\nearrow c(<-\lambda)$ as $r\nearrow +\infty$, then by \eqref{1} one has $\rho''\nearrow +\infty$ as $r\nearrow +\infty$. 
Since $\rho'<-\lambda$, we have a contradiction. 
Hence, we obtain $\rho'\nearrow -\lambda$ as $r\nearrow +\infty$. Therefore, $M$ is asymptotically flat.

Set $s:=-r$. Then, we have $\dot\rho(s)=-\rho'(r)<0$ and $\ddot\rho(s)>0$ $(s\in\mathbb{R}).$ 
Since $\lambda<\dot\rho<0$ and $\ddot\rho>0$, we have $\dot\rho\nearrow\beta$ as $s\nearrow+\infty$ for some constant $\beta\leq0$. If $\beta\not=0$, then since $\rho>0$, we have a contradiction. Hence, we have $\dot\rho\nearrow0$ as $s\nearrow+\infty$. Therefore, we have $\rho'\searrow0$ as $r\searrow-\infty$, that is, $R\searrow\lambda$ as $r\searrow-\infty.$
\\

\noindent
{\bf Case (1) of Theorem \ref{CSZ}.}
\\

In this case, we have two cases. (A) $\rho>0$ on $(0,+\infty)$, $\rho(0)=0$, $\rho'(0)=1$, and $\rho^{(\text{even})}(0)=0$, or (B) $\rho<0$ on $(0,+\infty)$, $\rho(0)=0$, $\rho'(0)=-1$, and $\rho^{(\text{even})}(0)=0$ (cf. Sections 1.4.4 and 4.3.4 in \cite{Petersen16})
\\

\noindent
{\bf Case 1 $\lambda=0$.}

\noindent
(A) $\rho>0$ on $(0,+\infty)$.

In this case, it is easy to see that the solutions to \eqref{1} are
\[
\rho(r)=2c_1^2\tanh\left\{\frac{1}{2}(c_1^2 r+c_1^2c_2 )\right\},
\]
where $c_1$ and $c_2$ are constants.
Since $\rho(0)=0,$ one has 
\[
c_1=0~\text{or}~c_2=0.
\]

If $c_1=0$, then one has $\rho(r)\equiv0$, which is a contradiction.

If $c_2=0$, then one has
\[
F(r)=4\log\left(\cosh\left(\frac{c_1^2}{2}r\right)\right)+C,
\]
where $C$ is a constant. Since $\rho'(0)=1$, we have $c_1^2=1$ and 
\[
F(r)=4\log\left(\cosh\left(\frac{r}{2}\right)\right)+C.
\]
It is Hamilton's cigar soliton.

By the same argument, Case (B) cannot occur.
\\

\noindent
{\bf Case 2 $\lambda>0$.}
\\

\noindent
(A) $\rho>0$ on $(0,+\infty)$.
\\

\noindent
(A)-(I) When $R>0$. Since $\rho'>-\lambda$ and $\rho''<0$, $\rho'\searrow \alpha$ as $r\nearrow +\infty$, for some constant $\alpha$. Assume that $\alpha<0$. Since $\rho>0$, we have a contradiction. Hence, $\alpha\geq0$. Therefore, $\rho'>0$ on $(0,+\infty)$.

Assume that $\rho\nearrow c^2$ as $r\nearrow+\infty$ for some constant $c$. Then one has $\rho'\searrow0$ as $r\nearrow +\infty$.
Hence, by \eqref{1}, $\rho''$ tends to $-\frac{\lambda}{2}c^2$ as $r\nearrow+\infty$. Since $\rho'>0$, we have a contradiction. Therefore, $\rho\nearrow +\infty$ as $r\nearrow+\infty.$ 
By \eqref{1}, $\rho''\searrow -\infty$ as $r\nearrow +\infty$, which is a contradiction, because $\rho'>0$ on $(0,+\infty)$.
\\

\noindent
(A)-(II) When $R<0$. Since $\rho'<-\lambda$ and $\rho>0$, we have a contradiction.
\\

\noindent
(B) $\rho<0$ on $(0,+\infty)$.
\\

\noindent
(B)-(I) When $R>0$. 
Since $\rho'>-\lambda$ and \eqref{1}, we have $\rho''>0$ on $(0,+\infty)$. Hence, $\rho'$ is monotonically increasing. Since $\rho<0$, there exists no point in $(0,+\infty)$ at which $\rho'=0$. Thus, $\rho'<0$ on $(0,+\infty)$. By \eqref{3}, $\rho'''>0$ on $(0,+\infty)$, which yields contradiction, because $\rho<0$ on $(0,+\infty)$.
\\

\noindent
(B)-(II) When $R<0$. Since $\rho'<-\lambda<0$ and \eqref{1}, we have $\rho''<0$ on $(0,+\infty)$. Hence, $\rho'$ is monotonically decreasing and $\rho\searrow-\infty$ as $r\nearrow+\infty$. By \eqref{3}, we also have $\rho'''<0$ on $(0,+\infty)$, which yields $\rho'\searrow-\infty$ as $r\nearrow+\infty$. By \eqref{1}, we have 
\[
2R'+\rho R=0.
\]
From this, $R<0$, and $\rho'<0$, we have
\[
2(R'-\lambda(\ln(-R))')=-\frac{1}{2}(\rho^2)'.
\]
Integrating both sides, we obtain
\[
2R-2\lambda\ln(-R)=-\frac{1}{2}\rho^2+C,
\]
where $C$ is a constant. Since $-R\nearrow+\infty$ as $r\nearrow+\infty$, for sufficiently large $\tilde r$, we have
\[
\ln(-R)<-R,
\]
on $(\tilde r,+\infty)$, which yields 
\[
2R+2\lambda R<-\frac{1}{2}\rho^2+C,
\]
on $(\tilde r,+\infty)$. Since $R=\rho'+\lambda$, the inequality is equivalent to
\[
\rho'<-\frac{1}{4(1+\lambda)}\rho^2+C,
\]
on $(\tilde r,+\infty)$. Since $-\rho^2\searrow-\infty$, for sufficiently large $\tilde r_1>\tilde r$, we have
\[
\rho'<-\frac{1}{100(1+\lambda)}\rho^2,
\]
on $(\tilde r_1,+\infty)$. Hence, we have
\[
(-\rho^{-1})'<-\frac{1}{100(1+\lambda)},
\]
on $(\tilde r_1,+\infty)$. Integrating both sides, we obtain
\[
-\rho^{-1}(r)+\rho^{-1}(\tilde r_1)<-\frac{1}{100(1+\lambda)}(r-\tilde r_1).
\]
Taking $r\nearrow+\infty$, we have a contradiction.
\\

\noindent
{\bf Case 3 $\lambda<0$.}
\\

\noindent
(A) $\rho>0$ on $(0,+\infty)$.
\\

\noindent
(A)-(I) When $R>0$. We have $\rho'>-\lambda>0$. By $\eqref{1}$, we have $\rho''<0$. 
Note that $\rho(0)=0$ and $\rho'(0)=1$. 
Hence, $-1<\lambda$ must be satisfied. By integrating both sides of $\rho'>-\lambda,$ one has 
$-\lambda r<\rho(r)$. Therefore, $\rho\nearrow +\infty$ as $r\nearrow +\infty$. Since $\rho'>-\lambda$ and $\rho''<0$, we have $\rho'\searrow\alpha(\geq-\lambda)$ as $r\nearrow+\infty$ for some constant $\alpha$. If $\alpha\not=-\lambda$, by \eqref{1}, we have that $\rho''$ tends to $-\infty$ as $r\nearrow+\infty$, which is a contradiction, because $\rho'>-\lambda$. Therefore, $\rho'\searrow-\lambda$ as $r\nearrow+\infty$, that is, $R\searrow0$ as $r\nearrow+\infty$.
\\
 
\noindent
(A)-(II) When $R<0$. We have $\rho'<-\lambda$. By \eqref{1}, we have $\rho''>0$. Since $\rho'(0)=1$, we have $\rho'>1$ on $(0,+\infty)$. Hence, $-1>\lambda$ must be satisfied. 
Since $\rho$ is a monotonically increasing convex function, $\rho\nearrow+\infty$ as $r\nearrow+\infty$. 
If $\rho'\nearrow c(<-\lambda)$ as $r\nearrow +\infty$, then by \eqref{1}, one has $\rho''\nearrow +\infty$ as $r\nearrow +\infty$. Since $\rho'<-\lambda$, we have a contradiction. 
Therefore, $\rho'\nearrow -\lambda$ as $r\nearrow +\infty$, and $M$ is asymptotically flat.
\\

\noindent
(B) $\rho<0$ on $(0,+\infty)$.
\\

\noindent
(B)-(I) When $R>0$. In this case, $\rho'>-\lambda>0$, which yields a contradiction, because $\rho'(0)=-1$.
\\

\noindent
(B)-(II) When $R<0$. We have $\rho'<-\lambda$. By \eqref{1}, we have $\rho''<0$. Since $\rho'(0)=-1$, $\rho'<-1$ holds. Integrating both sides of \eqref{1} from $0$ to $r$ yields 
\[
2\rho'+2+\frac{1}{2}\rho^2+\int_0^r\lambda\rho(s)ds=0.
\] 
Since $\lambda\rho>0$ on $(0,+\infty)$, we have
\[
2\rho'+\frac{1}{2}\rho^2<0.
\]
Hence, one has
\[
(-\rho^{-1})'<-\frac{1}{4}.
\]
Integrating this inequality from $\tilde r(>0)$ to $r$, we have
\[
-\rho^{-1}(r)+\rho^{-1}(\tilde r)<-\frac{1}{4}(r-\tilde r),
\]
which yields a contradiction, because $\rho<0$ on $(0,+\infty)$.
\end{proof}

\begin{remark}
For two-dimensional solitons, Theorem \ref{main2D} solves many problems (see for example Page 51 of \cite{Chowetal07}).
\end{remark}






\bibliographystyle{amsbook}

\end{document}